\documentclass[11pt]{amsart}
\usepackage[toc,page]{appendix}
\usepackage{amsmath,amssymb,latexsym}
\usepackage[small]{caption}
\usepackage{graphicx,color,mathrsfs,tikz}
\usepackage{subfigure,color}
\usepackage{cite}
\usepackage[colorlinks=true,urlcolor=blue,
citecolor=red,linkcolor=blue,linktocpage,pdfpagelabels,
bookmarksnumbered,bookmarksopen]{hyperref}
\usepackage[italian,english]{babel}
\usepackage{enumitem}
\usepackage[left=2.7cm,right=2.7cm,top=2.9cm,bottom=2.9cm]{geometry}
\usepackage[hyperpageref]{backref}

\usepackage[colorinlistoftodos,prependcaption]{todonotes}

\numberwithin{equation}{section}
\newtheorem{theorem}{Theorem}[section]
\newtheorem{proposition}[theorem]{Proposition}
\newtheorem{lemma}[theorem]{Lemma}

\newtheorem{definition}[theorem]{Definition}
\theoremstyle{definition}

\newtheorem{remark}[theorem]{Remark}
\renewcommand{\epsilon}{\eps}

\newcommand{\E}{{\mathcal E}}

\renewcommand{\H}{{\mathscr H}}
\renewcommand{\L}{{\mathcal L}}

\newcommand{\R}{{\mathbb R}}

\newcommand{\eps}{\varepsilon}

\newcommand{\M}{\mathscr{M}}

\newcommand{\pnorm}[2][]{\if #1'' \left|#2\right|_p \else \left|#2\right|_{#1} \fi}

\renewcommand{\theta}{\vartheta}

\newcommand{\Sf}{{\mathbb S^{n-1}}}
\newcommand{\Rn}{{\mathbb R^{n}}}
\newcommand{\eqlab}[1]{\begin{equation}  \begin{aligned}#1 \end{aligned}\end{equation}} 
\newcommand{\bgs}[1]{\begin{equation*} \begin{aligned}#1\end{aligned}\end{equation*}} 
 \newcommand{\syslab}[2] []  {\begin{equation}#1  \left\{\begin{aligned}#2\end{aligned}\right.\end{equation}} 
  \newcommand{\sys}[2][]{\begin{equation*}#1  \left\{\begin{aligned}#2\end{aligned}\right.\end{equation*}}
  
\def\Xint#1{\mathchoice
{\XXint\displaystyle\textstyle{#1}}%
{\XXint\textstyle\scriptstyle{#1}}%
{\XXint\scriptstyle\scriptscriptstyle{#1}}%
{\XXint\scriptscriptstyle\scriptscriptstyle{#1}}%
\!\int}
\def\XXint#1#2#3{{\setbox0=\hbox{$#1{#2#3}{\int}$ }
\vcenter{\hbox{$#2#3$ }}\kern-.6\wd0}}

\def\dashint{\Xint-}

\title[Asymptotic mean value properties for anisotropic operators]{Asymptotic mean value properties \\ for fractional anisotropic operators}

\author[C. Bucur]{Claudia Bucur}
\author[M.\ Squassina]{Marco Squassina}

\address[C. Bucur]{School of Mathematics and Statistics \newline\indent
	University of Melbourne \newline\indent
	813 Swanston Street, Parkville VIC 3010, Australia}
\email{c.bucur@unimelb.edu.au}

\address[M.\ Squassina]{Dipartimento di Matematica e Fisica \newline\indent
	Universit\`a Cattolica del Sacro Cuore \newline\indent
	Via dei Musei 41, I-25121 Brescia, Italy}
\email{marco.squassina@unicatt.it}

\thanks{The second author is member
	of {\em Gruppo Nazionale per l'Analisi Ma\-te\-ma\-ti\-ca, la Probabilit\`a e le loro Applicazioni} (GNAMPA) 
	of the {\em Istituto Nazionale di Alta Matematica} (INdAM)}



\begin{document}

\begin{abstract}
	We obtain an asymptotic representation formula for harmonic functions
	with respect to a linear anisotropic nonlocal operator. Furthermore we get a Bourgain-Brezis-Mironescu
	type limit formula for a related class of anisotropic nonlocal norms.
\end{abstract}
\maketitle

\begin{center}
	\begin{minipage}{8.5cm}
		\small
		\tableofcontents
	\end{minipage}
\end{center}

\medskip
\section{Introduction}
This paper presents an asymptotic mean value property for harmonic functions for a class of anisotropic nonlocal operators. To introduce the argument, we notice that as known from elementary PDEs facts, a $C^2$ 
function $u:\Omega\subset\R^n\to\R$ is harmonic in
$\Omega$ (i.e. it holds that $\Delta u=0$ in $\Omega$) if and only if it satisfies the mean value property, that is
\begin{equation*}
u(x)=\dashint_{B_r(x)} u(y)dy,\qquad \text{whenever  $B_r(x)\subset\Omega$}.
\end{equation*}
As a matter of fact, this condition can be relaxed to a pointwise formulation by saying that $u\in C^2(\Omega)$ satisfies $\Delta u(x)=0$ at a point $x\in\Omega$ if and only if
\eqlab{\label{media} u(x) =\dashint_{B_r(
x)} u(y)dy + \mathfrak o(r^2), \qquad \mbox{ as } r\to 0.}
This asymptotic formula holds true also in the viscosity sense for any continuous function. 
A similar property can be proved for quasi-linear elliptic operators such as the $p$-Laplace operator $-\Delta_p u$
in the asymptotic form, as the radius $r$ of the ball vanishes. More precisely, Manfredi, Parviainen and Rossi proved in \cite{MPR} that,
if $p\in (1,\infty]$, a continuous function $u:\Omega\subset\R^n\to\R$ is $p$-harmonic in $\Omega$ in viscosity sense if and only if 
\begin{equation}
\label{mediav}
\varphi(x)\geq (\leq)\,\frac{p-2}{2p+2n}\Big(\max_{\overline{B_r(x)}}\varphi+\min_{\overline{B_r(x)}}\varphi\Big)+\frac{2+n}{p+n}\,\dashint_{B_r(x)}\varphi(y)dy
+ \mathfrak o (r^2),
\end{equation}
for any $\varphi\in C^2$ such that $u-\varphi$ has a strict minimum (strict maximum for $\leq$) at $x\in\overline{\Omega}$ at the zero level.
Notice that formula \eqref {mediav} reduces to \eqref{media} for $p=2$. 
Formula \eqref{mediav} holds in the classical sense for smooth functions, at those points $x\in\overline{\Omega}$ such that $\nabla u(x)\neq 0$.
On the other hand, the case $p=\infty$ offers a counterxample for the validity 
of \eqref{mediav} in the classical sense, since the function $|x|^{4/3}-|y|^{4/3}$ 
is $\infty$-harmonic in $\R^2$ in the viscosity sense but \eqref{mediav} fails to hold pointwisely. If $p\in (1,\infty)$ and $n=2$
Arroyo and Llorente \cite{AL} (see also \cite{LM}) proved that the characterization holds in the classical sense. The limit case $p=1$
was finally investigated in \cite{MPR2}. 

Since the local (linear and nonlinear) case is well understood, it is natural to wonder 
about the validity of some kind of asymptotic mean value property in the {\em nonlocal} case. 
As a first approach, we want to investigate this type of property for a nonlocal, linear, anisotropic operator, defined as
	\eqlab{ \label{operator}
		\L u(x)= \int_0^\infty d\rho \int_{\Sf} da(\omega) \frac{\delta (u, x,\rho\omega)}{\rho^{1+2s}},
		}  
		where
		\[ 
		   \delta (u,x,y):= 2u(x)-u(x-y)-u(x+y).
		   \]
Here, $a$ is a non-negative measure on $\Sf$, finite i.e.
\eqlab{\label{finitemeas}  \int_{\Sf} d a \leq \Lambda }
		for some real number $ \Lambda >0$. We refer to this type of measure as \emph{spectral measure}, as it is common in the literature.
		We notice that when the measure $a$ is absolutely continuous with respect to the Lebesgue measure, i.e. when
		\[ da(\omega)= a(\omega) d\H^{n-1}(\omega)\]
		for a suitable, non-negative function $a\in L^1(\Sf)$,  the operator can be represented (using polar coordinates) 
		as
	\eqlab{\label{dos}
		\L u(x)= \int_\Rn \delta (u, x,y)a\left(\frac{y}{|y|} \right)  \frac{dy}{|y|^{n+2s}} .
		}
		Moreover, if $a \equiv 1$ then the formula gets more familiar, as we obtain the well-known fractional Laplace operator (see, e.g. \cite{hitch,nonlocal,Silvphd} and other references therein).  \\
		We remark also that the operator $\L$ is pointwise defined in an open set $\Omega\subset \Rn$ when, for instance, $u\in C^{2s+\eps}(\Omega)\cap L^\infty(\Rn)$. (Here, $C^{2s+\eps}(\Omega)$ denotes $C^{0,2s+\eps}(\Omega)$ or $C^{1,2s+\eps-1}(\Omega)$ for a small $\eps>0$, when $2s+\eps<1$, respectively when $2s+\eps\geq 1$.)\\
		As a matter of fact, the operator \eqref{operator} has been widely studied in the literature, being $\L$ the generator of any stable, symmetric Levy process. In particular, regularity issues for harmonic functions of $\L$ have been studied in papers like \cite{BassChen,sztonyk,basskass,bogdan1,kassmann} (check also other numerous references therein). There, some additional condition are required to the measure, in order to assure regularity. A typical assumption when $\L$ is of the form \eqref{dos} is 
		\bgs{
		0<c\leq a(y)\leq C \qquad \mbox{ in } \Sf,
		}
		or less restrictively
		\[ a(y)\geq c>0 \qquad \mbox{ in a subset of positive measure } \Sigma \subset \Sf.\] 
		Furthermore, for instance in \cite{BassChen} the measure needs not to be absolutely continuous with respect to the Lebesgue measure. In \cite{rosserr,rosval}, the optimal regularity is proved for general operators of the form \eqref{operator}, 
		with the ``ellipticity'' assumption 
	\eqlab{\label{elliptic}
		\inf_{\overline \omega\in \Sf} \int_\Sf |\omega \cdot \overline \omega|^{2s}\, da(\omega)\geq \lambda>0}
for some real number $\lambda$.
We note that the assumptions \eqref{elliptic} are satisfied by any stable operator with the spectral measure which is $n$-dimensional (that is, when the measure is not supported on any proper hyperplane of $\Rn$). 	 We will discuss some details related to this ellipticity requirement in Section \ref{weak} and in Remark \ref{classicalLaplace}. 
 \\

In this paper, for \eqref{operator} and \eqref{finitemeas}, we will adopt a potential theory approach, by using a ``mean kernel'', and provide a necessary and sufficient condition for a function to be  harmonic for $\L$, in the viscosity sense. To be more precise, 
we denote  for some $r>0$
	\bgs{
		\M^s_r u(x) :=
		 c(n,s,a) r^{2s} \int_r^\infty d\rho \int_\Sf da(\omega) 	\frac{u(x+\rho \omega) + u(x-\rho \omega)}{(\rho^2-r^2)^s\rho},		
	}
with
	\[ 
	 c(n,s,a) := \frac{ \sin \pi s}{\pi} \left(\int_\Sf da\right)^{-1}
	. 
	\]

\noindent The following is the main result of the paper.

\begin{theorem}\label{theorem} Let $\Omega\subset \Rn$ be an open set and $u\in L^\infty(\Rn)$. The asymptotic expansion
	\eqlab{ \label{eq1}
	 u(x)=  \M^s_r u(x) 
	 +\mathcal O(r^{2}), \qquad \mbox{ as } r\to 0
	 } 
	holds for all $x\in \Omega$ in the viscosity sense if and only if
	\[ 
	\L u(x) =0 
	\]
in the viscosity sense.
\end{theorem}
\noindent We point out here the paper \cite{FerrariMean}, where the author studies a general type of nonlocal operators defined by means of mean value kernels. Furthermore, the readers can check \cite{Landkof,bucur,AbatLarge} or the very nice recent monography \cite{garofalo} and other references therein, for details on the mean kernel and the mean value property for the fractional Laplacian.\\

As further results, we provide  some asymptotics for $s\nearrow1$ of the operator $\L$ and of the mean value $\M_r^s$. We also prove a Bourgain-Brezis-Mironescu type of formula, for a nonlocal norm related to the operator $\L$. In fact, as $s\nearrow1$, we obtain the ``integer'', local counterpart of the objects under study.
	\\
	
	This paper is organized as follows: in the next section we introduce some notations and some preliminary results. Section \ref{visc} deals with the viscosity setting and with the proof of Theorem \ref{theorem}. In Section \ref{weak} we make some remarks about the weak setting. In the last Section \ref{asymp-sec} we study the asymptotic behavior as $s\nearrow1$ of our fractional operators and prove a Bourgain-Brezis-Mironescu type of formula.
\section{Preliminary results and notations}

\noindent\textbf{Notations}\\
We use the following notations throughout this paper.
\begin{itemize}
\item For some $r>0$ and any $x\in \Rn$
\bgs{ & B_r(x):=\{ y\in \Rn \, \big| \, |x-y|<r\} , \qquad B_r:=B_r(0).\\
 &\Sf =\partial B_1.}
	 \item For  any $x>0$, the Gamma function is (see \cite{abramowitz}, Chapter 6):
\bgs{ \Gamma(x)=\int_0^{\infty}t^{x-1}e^{-t}\, dt.\label{gamma}
}
\item For  any $x,y>0$, the Beta function is (see \cite{abramowitz}, Chapter 6):
\bgs{ \beta(x,y)=\int_0^{\infty}\frac{t^{x-1}}{(t+1)^{x+y}}\, dt.\label{beta}
}	
	\end{itemize}

\noindent We remark as a first thing the following integral indentity. 
\begin{lemma} \label{lem1}
For any $r>0$
\[ \frac{2\sin \pi s}{\pi} r^{2s} \int_r^\infty \, \frac{d\rho }{(\rho^2 -r^2)^s \rho}  
 =1. \label{Ir} \]
\end{lemma}
\begin{proof}
Changing coordinates, we get that
	\bgs{
		r^{2s} \int_r^\infty \, \frac{d\rho }{(\rho^2 -r^2)^s \rho} = \frac12 \int_0^\infty \frac{dt}{t^s(t+1)} =\frac{\beta(1-s,s)}2
	= \frac{\pi}{2 \sin(\pi s)},
	}
where we have used formulas 6.2.2 and 6.1.17 in \cite{abramowitz}. 
\end{proof}

\noindent
We obtain the asymptotic mean value property for smooth functions, as follows.
\begin{theorem}\label{asymp}
	Let $u\in C^{2}(\Omega)\cap L^\infty(\Rn)$. Then 	
	\bgs{ u(x)= \M_r^s u(x) +
		  c(n,s,a) r^{2s} \L u(x)	  +\mathcal O(r^{2}) 
	}
	as $r\to 0$.
\end{theorem}

\begin{proof}
We fix $x\in \Omega$ and $\delta>0$ such that $B_{2\delta}(x)\subset \Omega$. For any $0<r<\delta$, by Lemma \ref{lem1} we have that
\bgs{\label{first} 
	{u(x)- \M_r^s u(x)}
	=&\, c(n,s,a) r^{2s} \int_r^\infty d\rho \int_\Sf da(\omega)\frac{ 2u(x)}{(\rho^2-r^2)^s\rho}- \M_r^s u(x)
		\\
	  = &\, {c(n,s,a)}r^{2s}\int_r^\infty d\rho \int_\Sf da(\omega)  \frac{\delta(u,x,\rho\omega)}{(\rho^2 -r^2)^s \rho}
	  \\
	  =&\, c(n,s,a) \int_1^\infty d\rho \int_\Sf da(\omega)  \frac{\delta(u,x,r \rho\omega)}{(\rho^2 -1)^s \rho}. 
	  }
Notice that $1<\delta/r$, so we write
\eqlab{\label{sec} \frac{u(x)- \M_r^su(x)}{c(n,s,a)}&= 
		 \int_{\frac{\delta}{r}}^\infty d\rho \int_\Sf da(\omega)  \frac{\delta(u,x,r \rho\omega)}{(\rho^2 -1)^s \rho} 
		 +
		 \int_1^{\frac{\delta}{r}} d\rho \int_\Sf da(\omega)  \frac{\delta(u,x,r \rho\omega)}{(\rho^2 -1)^s \rho} 
		 	\\
		&=: \mathcal I_1+ \mathcal I_2.
	}
We have that
\bgs{ 
	\mathcal I_1 &= 
	 	\int_{\frac{\delta}{r}}^\infty d\rho \int_\Sf da(\omega)  \frac{\delta(u,x,r \rho\omega)}{ \rho^{1+2s}} \frac{1} {\left(1-\frac{1}{\rho^2}\right)^s} .
	}
Denote by 
	\[t:= \frac{1}{\rho}\in \left(0,\frac{r}{\delta}\right) \] and for $r$ small enough with a Taylor expansion we have  that
	\[ (1-t^2)^{-s}= 1+ st^2 + \mathfrak o (t^2).\]
	So
\bgs{ 
	\mathcal I_1 = 
	&	\; 	\int_{\frac{\delta}{r}}^\infty d\rho \int_\Sf da(\omega)  \frac{\delta(u,x,r \rho\omega)}{ \rho^{1+2s}}
	 + s 	\int_{\frac{\delta}{r}}^\infty d\rho \int_\Sf da(\omega)  \frac{\delta(u,x,r \rho\omega)}{ \rho^{3+2s}} \\
	&\; + 
	\int_{\frac{\delta}{r}}^\infty d\rho \int_\Sf da(\omega)  \delta(u,x,r \rho\omega)
	    \mathfrak o \left({\rho^{-3-2s}}\right).
	}
Notice that
 \bgs{ 
 	\left| \int_{\frac{\delta}{r}}^\infty d\rho \int_\Sf da(\omega)  \frac{\delta(u,x,r \rho\omega)}{ \rho^{3+2s}}\right|  
 	  	\leq & \, 4\|u\|_{L^\infty(\Rn)} \int_\Sf da \int_{\frac{\delta}{r}}^\infty \frac{d\rho}{\rho^{3+2s}}
 	  	\\
 	  	\leq &\,  2\Lambda \|u\|_{L^\infty(\Rn)}\frac{r^{2+2s}\delta^{-2-2s}}{1-s} .  	
 	  	}
 	  	 	It follows that
\bgs{ 
	\mathcal I_1 = &\;		\int_{\frac{\delta}{r}}^\infty d\rho \int_\Sf da(\omega)  \frac{\delta(u,x,r \rho\omega)}{ \rho^{1+2s}} +\mathcal O( r^{2+2s}).
	}	
	On the other hand we write
\bgs{ 
	\mathcal I_2 &=
	 \int_1^{\frac{\delta}{r}} d\rho \int_\Sf da(\omega)  \frac{\delta(u,x,r \rho\omega)}{ \rho^{1+2s}} \\
	 &
	+ \int_1^{\frac{\delta}{r}} d\rho \int_\Sf da(\omega)  \frac{\delta(u,x,r \rho\omega)}{ \rho^{1+2s}}
	\left( \frac{\rho^{2s}}{(\rho^2-1)^s} -1\right) .
		}
		Then putting together $\mathcal I_1$ and $\mathcal I_2$ into \eqref{sec} we have that
		\eqlab{\label{aaar}
	\frac{u(x)- \M_r^su(x)}{c(n,s,a)}
				=&\; 	 \int_1^\infty d\rho \int_\Sf da(\omega)  \frac{\delta(u,x,r \rho\omega)}{ \rho^{1+2s}} \\
	   	&\;
	   	 	+ \int_1^{\frac{\delta}{r}} d\rho \int_\Sf da(\omega)  \frac{\delta(u,x,r \rho\omega)}{ \rho^{1+2s}}
	\left( \frac{\rho^{2s}}{(\rho^2-1)^s} -1\right)
	  +\mathcal O( r^{2+2s})
	  \\ 
	  =&\; 	 \int_0^\infty d\rho \int_\Sf da(\omega)  \frac{\delta(u,x,r \rho\omega)}{ \rho^{1+2s}} \\
	   	&\;
	   	 	+ \int_1^{\frac{\delta}{r}} d\rho \int_\Sf da(\omega)  \frac{\delta(u,x,r \rho\omega)}{ \rho^{1+2s}}
	\left( \frac{\rho^{2s}}{(\rho^2-1)^s} -1\right) 
	\\
	&\; -  	 \int_0^1 d\rho \int_\Sf da(\omega)  \frac{\delta(u,x,r \rho\omega)}{ \rho^{1+2s}}  +\mathcal O( r^{2+2s})
	  \\
	  =: &\;  r^{2s} \L u(x) + \mathcal J +\mathcal O( r^{2+2s}).
	  	  }
Recalling that $u\in C^{2}(\Omega)$,  both for $\rho \in(1,\delta/r)$  and for $\rho\in(0,1)$ we have that
	\[ 
		|\delta(u,x,r \rho\omega)|\leq  r^2 \rho^{2} \|u\|_{C^2(B_\delta(x))}.
		\]	
We thus obtain
\bgs{
	\left| \mathcal J \right| 
	\leq &\; \|u\|_{C^2(\Omega)} r^{2}  \int_\Sf da(\omega) \bigg( \int_1^{\frac{\delta}{r}} d\rho\,  \rho^{1-2s} \left( \frac{\rho^{2s}}{(\rho^2-1)^s} -1\right)\,  + \int_0^1 d\rho\, \rho^{1-2s} \bigg)
		\\
	=& C\frac{ r^2 (1+ \mathcal O (r^{2s}) ) }{2(1-s)} .
		}
		The last line follows since
		\[
		\int \rho^{1-2s} \left( \frac{\rho^{2s}}{(\rho^2-1)^s} -1\right)\, d\rho = 
		- \frac{\rho^{2-2s}-(\rho^2-1)^{1-s}}{2-2s}.
		\]
		  Hence in \eqref{aaar} we finally get that
  	\[
 	{u(x)-\M^s_r u(x)}=  {c(n,s,a)} r^{2s}  \L u(x)+\mathcal O(r^{2}) . 	
 	\]	
	This concludes the proof of the theorem.
\end{proof}

\section{Viscosity setting}\label{visc}

We begin by giving the definitions for the viscosity setting. First of all (as in \cite{caffy}), we define the notion of viscosity solution.
\begin{definition} A function $u\in L^\infty(\Rn)$, lower (upper) semi-continuous in $\overline \Omega$ is a viscosity supersolution (subsolution) to 
	\[
		\L u=0, \qquad \mbox{ and we write } \quad \L u \leq \,(\geq) \,  0
	\]
	if for every $x\in\Omega$, any neighborhood $U=U(x)\subset \Omega$ and any $\varphi \in C^2(\overline U)$ such that
	\bgs{ 
		& \varphi (x)=u(x)
		\\
		& \varphi(y)< \, (>) \,u(y), \quad \mbox{ for any } y\in U\setminus \{x\},
		}
	if we let
		\syslab[v =]{&\varphi , \quad \mbox{ in }  U	\\
							&u, \quad \mbox{ in } \Rn\setminus U	
							\label{vvv} 
							}
				then
				 \bgs{
				 	\L v (x) \geq \,(\leq)\,0.
				 }			
	A viscosity solution of $\L u =0$ is a (continuous) function that is both a subsolution and a supersolution. 
\end{definition}	

\begin{definition} \label{meanvisc} A function $u\in L^\infty(\Rn)$, lower (upper) semi-continuous in $\overline \Omega$, verifies for any $x\in \Omega$ 
 	\[
		u(x)= \M_r^s u(x)+ \mathcal O(r^{2})
 	\]
	as $r\to 0$, in a viscosity sense, if for any 
	neighborhood $U=U(x)\subset \Omega$ and any $\varphi \in C^2(\overline U)$ such that
	\bgs{ 
		& \varphi (x)=u(x)
		\\
		& \varphi(y)<\, (>)\,  u(y), \quad \mbox{ for any } y\in U\setminus \{x\},
		}
if we let
		\sys[v=]{&\varphi, \quad \mbox{ in }  U	\\
							&u , \quad \mbox{ in } \Rn\setminus U	}
				then
	\eqlab{ \label{opium2}
		v(x) \geq \, (\leq) \, \M_r^s v(x)+ \mathcal O(r^{2}).
 		}
\end{definition}	

\noindent We can now prove the main theorem of this paper.
\begin{proof}[Proof of Theorem \ref{theorem}]
Let $x\in \Omega$ and any $R>0$ be such that $\overline {B_R(x)}\subset \Omega$. Let $\varphi \in C^2(\overline{B_R(x)})$ be such that
\bgs{ 
		& \varphi (x)=u(x)
		\\
		& \varphi(y)< \, u(y), \quad \mbox{ for any } y\in B_R(x)\setminus\{x\}.
		}
		We let $v$ be defined as in \eqref{vvv}, hence $v\in C^2({B_R(x)})\cap L^\infty(\Rn)$.	
		By Theorem \ref{asymp} we have that 
		\eqlab{ \label{asaf}
			v(x)= \M_r^sv(x) +  c(n,s,a) r^{2s} \L v(x) + \mathcal O(r^{2}).
			}
			We prove at first that if $u$ satisfies \eqref{eq1} in the viscosity sense given by Definition \ref{meanvisc} then $u$ is a supersolution of
$\L u(x)=0$ in the viscosity sense.
Since
\[
			 v(x)\geq \M_r^s v(x)+ \mathcal O (r^{2})
		\]
		dividing by $r^{2s}$ in \eqref{asaf} and sending $r\to 0$, it follows that 
		\[
			\L v(x)\geq 0.
		\]
		At the same manner, one proves that $u$ is a subsolution of $\L u(x)=0$ in the viscosity sense.
		
In order to prove the other implication, if $u$ is a supersolution, one has from \eqref{asaf} that
	\[
		\limsup_{r\to 0} \frac{v(x)-\M_r^s v(x)}{r^{2s}}  \geq 0,
		\]
		hence \eqref{opium2}. In the same way, one gets the conclusion when $u$ is a subsolution.
\end{proof}
%

\section{Some remarks about the weak setting}\label{weak}
In this section we consider the weak setting, and in particular provide the condition for a weak solution to be a  pointwise, and a viscosity solution. In this sense, Theorem \ref{theorem} applies also to weak solutions.
We consider here $\Omega$ to be an open bounded set, with $C^2$ boundary.\\

\noindent
Let us now define the following norms, semi-norms and spaces:
\[
	[u]_{H^1_a(\Rn)}:=\left(\int_{\Rn}dx\int_\Sf da(\omega) \left(\nabla u(x)\cdot \omega\right)^2\right)^{\frac12}, \qquad 
\|u\|_{H^1_a(\Rn)} := [u]_{H^1_a(\Rn)} + \|u\|_{L^2(\Rn)},
\]
\bgs{
 &H_a^1(\Rn):=\left\{ u\in L^2(\Rn)\, \big| \,\, [u]_{H^1_a(\Rn)} <+\infty\right\},
\qquad 
&H_{a,0}^1(\Rn):= \overline{C^\infty_c(\Rn)}^{\|\cdot\|_{H_a^1(\Rn)}},
}
and 
\[
	[ u]_{H^s_a(\Rn)}:=\left(\int_\Rn dx \int_\R d\rho \int_\Sf da(\omega) \frac{\left(u(x)-u(x+\rho\omega)\right)^2}{|\rho|^{1+2s}} \right)^{\frac12}
 	.\]
Taking into account that
$$
\|v\|:=\left(\int_\Sf (v\cdot \omega)^2da(\omega) \right)^{\frac12},\quad v\in\R^n,
$$ 	
is a norm in $\R^n$, it is readily seen that $H_a^1(\Rn)$ is a Banach space.

We take the operator $\L$ that satisfies the ellipticity assumption \eqref{elliptic} and we consider weak solutions of the equation 
\[ \L u=0 \quad \mbox{ in } \Omega.\]
In particular, we take  $u\in L^\infty(\Rn)$  of finite energy, i.e. such that
\eqlab{\label{finiteenergy}
[u]_{H^s_a(\Rn)}<\infty,
}
and look for critical points of the energy 
\[ 
\E(u):=\frac{1}{4}\int_\Rn dx \int_\R d\rho \int_\Sf da(\omega) \frac{\left(u(x)-u(x+\rho\omega)\right)^2}{|\rho|^{1+2s}} .
\]
Then for any $\varphi \in C^\infty_c(\Omega)$, we compute formally 
\[ 
\E'(u)[\varphi]= \frac{1}{2}\int_\Rn \,dx \int_\R d\rho \int_\Sf da(\omega) \frac{\big(u(x)-u(x+\rho\omega) \big) (\varphi(x) -\varphi(x+\rho\omega) \big)}{|\rho|^{1+2s}} .
\]
So we say that $u\in L^\infty(\Rn)$  of finite energy is a weak solution of 
\[ \L u=0 \quad \mbox{ in } \Omega\]
if and only if
\eqlab{\label{weakform}
\E'(u)[\varphi]=0,\quad \forall\varphi \in C^\infty_c(\Omega). 
}

\begin{remark} \label{weakprop}
Let us justify the fact that if $u\in L^\infty(\Rn)\cap C^2(\Omega)$, of finite energy satisfies \eqref{weakform}, then $\L u=0$  in $\Omega$ pointwisely.     
We have that
\bgs{
	\E'(u)[\varphi]= &\, \frac{1}{2}\ \int_\Rn \,dx \int_\R d\rho \int_\Sf da(\omega) \frac{\big(u(x)-u(x+\rho\omega) \big) \varphi(x) }{|\rho|^{1+2s}} 
	\\
	&\, - \frac{1}{2}\int_\Rn \,dx \int_\R d\rho \int_\Sf da(\omega) \frac{\big(u(x)-u(x+\rho\omega) \big) \varphi(x+\rho\omega) }{|\rho|^{1+2s}}  
	\\
	= &\, \frac{1}{2} \int_\Rn \,dx \int_\R d\rho \int_\Sf da(\omega) \frac{\big(2u(x)-u(x+\rho\omega) -u(x-\rho\omega)\big) \varphi(x) }{|\rho|^{1+2s}} 
	\\
	=&\, \int_\Rn \,dx \,\L u(x)\varphi(x) ,
	}
	where we have used the change of variable $y=x+\rho\omega$ and  the symmetry in $\rho$.
So indeed $\E'(u)[\varphi]=0$ for any $\varphi\in C^\infty_c(\Omega)$ implies that $\L u= 0$ almost everywhere in $\Omega$.
Furthermore, proceeding  as in  \cite[Proposition 2.1.4]{Silvphd} we have that $\L u$ is continuous in $\Omega$, hence $\L u=0$ pointwise in $\Omega$.
\end{remark}

The following is the main result of this section. 
\begin{theorem} Let $u\in L^\infty(\Rn)$ of finite energy be a weak solution of 
\[ \L u =0 \qquad \mbox{ in }\;  \Omega.\] Then for any $\Omega'\subset\subset \Omega$ we have that $u\in C(\overline{\Omega'})$ is a viscosity solution of 
\[ \L u=0 \qquad \mbox{ in }\; \Omega'.\]
\end{theorem}

\begin{proof}
First of all notice that thanks to  \cite[Theorem 1.1, a)]{rosserr} we have that $u\in C(\overline{\Omega'})$. Using the approach in Theorem \cite[Theorem 1]{Valdy}, we do the following. 
We consider a sequence of mollifiers of $u$, more precisely we take $\varphi \in C^\infty_c(\Omega')$ and 
\[ \varphi_\eps(x)=\frac{1}{\eps^n} \varphi\left(\frac{x}{\eps}\right), \qquad u_\eps(x)=u*\varphi_\eps(x)=\int_{\Rn} u(x-y)\varphi_\eps(y)\, dy.\]
The basic properties of mollifiers give us that
\eqlab{\label{molly} & u_\eps\in C^\infty(\Rn) &\\
		& u_\eps \xrightarrow[\, \, \eps \to 0\, \,  ]{} u   &&\mbox{ a.e. in } \Rn \\
		&u_\eps  \xrightarrow[\, \, \eps \to 0\, \, ]{} u   &&\mbox{ locally uniformly in } \Omega'.} 
		Furthermore we have that (for $\eps$ small enough) in the weak sense
		\eqlab{\label{weak11} \L u_\eps=0 \quad \mbox{ in } \Omega'.}
		Indeed for any $\psi \in C^\infty_c(\Omega')$, using Fubini we get that 
		\bgs{ \E'( u_\eps)[\psi]=&\, \int_{\Rn} dx \int_\R d\rho \int_\Sf da(\omega) \frac{\big(u_\eps(x)-u_\eps(x+\rho\omega)\big)\big(\psi(x)-\psi(x+\rho\omega)\big) }{|\rho|^{1+2s}}
		\\
		=&\,   \int_{\Rn} dx \int_\R d\rho \int_\Sf da(\omega)  \frac{ \psi(x)-\psi(x+\rho\omega) }{|\rho|^{1+2s}} \int_\Rn dz \,\big(u(z)-u(z+\rho\omega)\big)\varphi_\eps(x-z)
		\\
		=&\;  \int_\Rn dz  \int_\R d\rho \int_\Sf da(\omega) \frac{u(z)-u(z+\rho\omega)}{|\rho|^{1+2s}}\int_\Rn dx\, \big(\psi(x)-\psi(x+\rho\omega) \big)\varphi_\eps(x-z)
		\\
		=&\;   \int_\Rn dz  \int_\R d\rho \int_\Sf da(\omega) \frac{\big(u(z)-u(z+\rho\omega)\big) \big(\psi_\eps(z)-\psi_\eps(z+\rho\omega) \big) }{|\rho|^{1+2s}}
		\\=&\;  \E'( u)[\psi_\eps],
		}
		where 
		\[\psi_\eps(z)=\int_{\Rn} \psi(x) \varphi_\eps(x-z) \, dx.\]
		Notice that (for $\eps$ small) $\psi_\eps \in C^\infty_c(\Omega)$. Since $u$ is a weak solution of $\L u=0$ in $\Omega$, the claim \eqref{weak11} follows.
		Recalling that $u_\eps$ is smooth in $\Omega'$, we know that $\L u_\eps (x)$ is well defined for any $x\in \Omega'$. According to Remark \ref{weakprop} we get that pointwise in $\Omega'$ 
\bgs{\L u_\eps(x) = 0.}
 	Taking $v$ that touches $u_\eps$ at $x_0 \in \Omega'$ from below (as defined in \eqref{vvv}), i.e. $u_\eps(x_0)=v(x_0) $, with $v -u_\eps\leq 0 $ in $\Rn $ we get that
 	 	\bgs{
 	 	\L v(x_0)= \int_0^\infty d\rho \int_\Sf da(\omega)\frac{2v(x_0)-v(x_0+\rho \omega) -v(x_0-\rho\omega) }{|\rho|^{1+2s}}   \geq  \L u_\eps(x_0)=0.}
 	 	This proves that $u_\eps$ is a supersolution. In the same way, one can prove that $u_\eps$ is a subsolution, thus
 	 	\[ \L u_\eps=0\] also in the viscosity sense. It is enough now to observe that the operator $\L$ satisfies the first two conditions of \cite[Definition 3.1]{caffy} (one can prove the second item as in \cite[Proposition 2.1.4]{Silvphd}). 
 	 	Taking into account \eqref{molly}, we can use  \cite[Corollary 4.6]{caffy} to conclude that 
 	 	\[ \L u =0\] in $\Omega'$ in the viscosity sense. 
 	 	\end{proof}
 
 Also, we notice that if one takes $u\in L^{\infty}(\Rn)\cap C^\alpha(\Rn)$ for some $\alpha>0$ such that $\alpha+2s$ is not an integer, according to  \cite[Theorem 1.1, b)]{rosserr} 
we  get that $u\in C^{\alpha+2s}(\Omega')$ for any $\Omega'\subset\subset \Omega$. 
 As remarked in \cite{rosserr}, one cannot remove the hypothesis that $u\in C^\alpha(\Rn),$ in order to obtain the $C^{2s+\alpha}$ regularity of $u$. So, in this way, a weak solution of $\L u=0$ in $\Omega$ is a both viscosity  and  pointwise solution in  $\Omega'$.

\section{Asymptotics as $s\nearrow1$}\label{asymp-sec}

In this section we provide some asymptotic properties 
on the operator $\L$, the mean value defined by $\M_r^s $ and the semi-norm in \eqref{finiteenergy}. 
We study their limit behavior as $s$ approaches the upper value $1$. 
Indeed, re-normalizing (multiplying by $(1-s)$) and sending $s\nearrow1$, 
we obtain the local counterpart of the operators under study. It is interesting in our opinion, from this point of view, to understand what is the influence of the non symmetric measure $da$  in the limit, with respect to having the $d\H^{n-1}$ measure on the hypersphere.

\vskip2pt
\noindent
 We begin by showing the following. 
 
\begin{proposition}Let $u\in C^2(\Omega)\cap L^\infty(\Rn)$. Then for all $x\in \Omega$
 \bgs{ \label{aaa2}
  		\lim_{s\nearrow 1} (1-s) \L u(x) =  - \frac12
  		\sum_{i,j=1}^n \left( \int_\Sf da(\omega) \omega_i\omega_j \right)\partial_{ij}^2 u(x).
  		}
		\end{proposition}
\begin{proof}
We fix $x\in \Omega$. 
By a Taylor expansion, we have that for every $\rho>0$ (such that $B_\rho(x)\subset \Omega$) and every $\omega\in \Sf$, there exists $\underline h:=\underline h(\rho,\omega), \overline h :=\overline h (\rho,\omega) \in [0,\rho]$ such that	
	\bgs{
	& \delta(u,x,\rho\omega) = - \frac{\rho^2}2 \langle D^2u(x+\overline h\omega)\omega,\omega\rangle
	- \frac{\rho^2}2 \langle D^2u(x-\underline h\omega)\omega,\omega\rangle.
	}
Since $u \in C^2(\Omega)$, we have that for any $\eps>0$ there exists $r:=r(\eps)>0$ such that
\eqlab{\label{c2cond}
\Big |\big\langle 
 (D^2u(x+h\omega)- D^2u(x))\omega, \omega\big \rangle \Big|\leq \Big|D^2u(x+h\omega)- D^2u(x)\Big||\omega|^2 <\eps,
 \\
 \qquad  \mbox{ whenever } |h|=|h\omega|\leq \rho<r 
 .} 
Fixing an arbitrary $\eps$ and taking the corresponding $r:=r(\eps)$,  we write
\bgs{
 	\L u(x) = &\; \int_0^r d\rho \int_\Sf da(\omega) \frac{\delta(u,x\rho\omega)}{\rho^{1+2s}} +  \int_r^\infty d\rho \int_\Sf da(\omega) \frac{\delta(u,x\rho\omega)}{\rho^{1+2s}} 
\\
=&\; -\frac12 \int_0^r d\rho \int_\Sf da(\omega) \rho^{1-2s} 
\big\langle 
 D^2u(x+\overline h\omega)\omega, \omega\big \rangle  
 \\ &\; -\frac12 \int_0^r d\rho \int_\Sf da(\omega) \rho^{1-2s} 
\big\langle 
 D^2u(x-\underline h\omega)\omega, \omega\big \rangle\\
 &\; +  \int_r^\infty d\rho \int_\Sf da(\omega) \frac{\delta(u,x\rho\omega)}{\rho^{1+2s}} 
 \\
=&\;: \left(-\frac12\right) \left(I_{r,s}^1+ I_{r,s}^2 \right)+ J_{r,s}.
} 
Now notice that 
\bgs{ 
I_{r,s}^1
				=&\; 
		\int_0^r d\rho \int_\Sf da(\omega) {\Big\langle \big(D^2u(x+\overline h\omega)-D^2u(x)\big)\omega, \omega\Big\rangle}{\rho^{1-2s}}\\
		+&\;  
		\int_0^r d\rho \int_\Sf da(\omega) {\langle D^2u(x)\omega, \omega\rangle}{\rho^{1-2s}}
	.		}
		By using \eqref{c2cond} we notice that
	\bgs{ \left|\int_0^r d\rho \int_\Sf da(\omega) {\Big\langle \big(D^2u(x+\overline h\omega)-D^2u(x)\big)\omega, \omega\Big\rangle}{\rho^{1-2s}}\right| \leq&\;\int_0^r d\rho \int_\Sf da(\omega) 	\eps\rho^{1-2s} \\
			\leq &\; \eps\Lambda \frac{r^{2-2s}}{2(1-s)}.
	}
	On the other hand, we get that
\eqlab{ \label{firr}
	\int_0^r d\rho \int_\Sf da(\omega) {\langle D^2u(x)\omega, \omega\rangle}{\rho^{1-2s}}
	=  &\;\sum_{i,j=1}^n {\partial_{ij}^2 u}(x) \int_0^r  d\rho \: \rho^{1-2s} \int_\Sf da(\omega) \omega_i \omega_j 
	\\
	=&\; \frac{r^{2-2s}}{2(1-s)} \sum_{i,j=1}^n {\partial_{ij}^2 u}(x) \int_\Sf da(\omega) \omega_i \omega_j
	\\
	=&\, \frac{r^{2-2s}}{2(1-s)} \sum_{i,j=1}^n m_{ij}{\partial_{ij}^2 u}(x)  , 
		}
using the notation 
  \bgs{
  	m_{ij}=\int_\Sf \omega_i\omega_j da(\omega)
.}
Multiplying by $(1-s)$, letting $s\nearrow1$ we get that
\[\lim_{s\nearrow 1} (1-s) I_{r,s}^1 = \frac12 \sum_{i,j=1}^n m_{ij} \partial_{ij}^2u(x) +\mathcal O(\eps).\]
In the same way, one gets the same limit for $I_{r,s}^2$.
 		   Notice also that for $s$ close to $1$ (hence when for instance $s>1/2$)
\bgs{
	\left |J_{r,s} \right | \leq  \frac{2 r^{-2s}\|u\|_{L^{\infty}(\Rn)}\Lambda}s  . 
}
Thus we obtain
\[ \lim_{s\nearrow 1} (1-s)J_{r,s}=0.\]
Using the arbitrariness of $\eps$, it follows  that 
	\bgs{
		\lim_{s\nearrow 1} (1-s)  \L u(x) = 
		-\frac{1}{2} \sum_{i,j=1}^n m_{ij}{\partial_{ij}^2 u}(x) 
		}
		hence the conclusion.
\end{proof}

The interested reader can check also Section 3 (and the Appendix U) in \cite{gettinacq} for the asymptotics as $s\to 1$ of another (general) type of nonlocal operator. 
\begin{remark}\label{classicalLaplace}
Notice that the matrix associated to the local operator, given by the constant coefficients 
\[ m_{i,j}= \int_\Sf da(\omega) \omega_i\omega_j \quad i,j=1,\dots,n \]
is symmetric. Then that the local operator that we have obtained, i.e.
\[ \mathcal I u:=-\frac{1}{2} \sum_{i,j=1}^n m_{ij}{\partial_{ij}^2 u}(x) \] is the classical Laplacian, up to a change of coordinates, provided that the matrix is also positive definite. In fact, in order to have this, one should ask that
\eqlab{\label{el11} \inf_{\overline \omega \in \Sf} \int_\Sf |\omega \cdot \overline \omega|^2 \geq  \lambda > 0.}
In that case, indeed for any $\overline \omega\in \Sf$ we have that
\[\sum_{i,j=1}^n  \overline \omega_i\overline\omega_j \int_\Sf da(\omega) \omega_i\omega_j = \sum_{i,j=1}^n \int_\Sf da(\omega)  \omega_i\overline \omega_i \omega_j \overline\omega_j = \int_\Sf da(\omega)  |\omega\cdot  \overline \omega|^2.\] 
Notice that \eqref{el11} is true if the ellipticity assumption \eqref{elliptic}  holds, uniformly in $s$.\\
\end{remark}

	\noindent Furthermore, we have the next result.
\begin{proposition}
Let $u\in C^1(\Omega) \cap L^\infty(\Rn)$. For all $x\in \Omega$ and any $r>0$ with $B_{2r}(x)\subset \Omega$ 
	\bgs{ 
		\lim_{s\nearrow 1}  \M^s_r u(x) =  \frac12\left(\int_\Sf da\right)^{-1}   \int_\Sf da(\omega) \left( u (x-r\omega) +u(x+r\omega)\right). 
	}
\end{proposition}

\begin{proof}
 We fix $\eps\in(0,1)$, which we will take arbitrarily small in the sequel.  We have that
 	\eqlab{ \label{opium}
 		\frac{\M_r^s u(x) }{ c(n,s,a)}= &\, r^{2s}\int_r^\infty d\rho\int_\Sf da(\omega) \frac{u(x+\rho\omega)+ u(x-\rho\omega) }{(\rho^2-r^2)^s\rho}
 		\\ 
 		=&\,	\int_1^\infty d\rho\int_\Sf da(\omega) \frac{u(x+r\rho\omega)+ u(x-r\rho\omega)}{(\rho^2-1)^s\rho}  
			 \\
			 =& \, \int_{1+\eps}^\infty d\rho\int_\Sf da(\omega) \frac{u(x+r\rho\omega)+ u(x-r\rho\omega)}{(\rho^2-1)^s\rho} \\
			 &\,+\int_1^{1+\eps} d\rho\int_\Sf da(\omega)  \frac{u(x+r\rho\omega)+ u(x-r\rho\omega)}{(\rho^2-1)^s\rho}
			  \\
			 =:&\, \mathcal I_1+ \mathcal I_2.
  	}
Now
	\bgs{ 
		\left| \mathcal I_1 \right| 			
			\leq &\, \int_{1+\eps}^2d\rho \int_\Sf da(\omega)\frac{|u(x+r\rho \omega)+u(x-r\rho\omega)|}{(\rho^2-1)^s\rho}
			\\
			&\; +
			\int_2^{{\infty}}d\rho \int_\Sf da(\omega)\frac{|u(x+r\rho \omega)+u(x-r\rho\omega)|}{(\rho^2-1)^s\rho}
			 \\
			\leq &\, 2 \|u\|_{C(B_{2r}(x))} \int_\Sf da 
			\int_{1+\eps}^2 \frac{d\rho}{(\rho-1)^s(\rho+1)^s\rho}
			\\
			&\,						+ 2 \|u\|_{L^\infty(\Rn)}\int_\Sf da \int_2^\infty\frac{d\rho}{(\rho^2-1)^s\rho}  
						\\
						\leq &\,    \frac{2\Lambda \|u\|_{C(B_{2r}(x))}}{(1+\eps)(2+\eps)^s} \int_{1+\eps}^2 \frac{d\rho}{(\rho-1)^s} 
						+ 4 \Lambda \|u\|_{L^\infty(\Rn)} \int_2^\infty \frac{d\rho}{ \rho^{1+2s}}
						\\
						\leq &\, \frac{2\Lambda \|u\|_{C(B_{2r}(x))} (1-\eps^{1-s})}{ 1-s} 
						+\frac{4 \Lambda \|u\|_{L^\infty(\Rn)}}s
						.	
			}	
Notice that since
	\[
	 \lim_{s\nearrow1} \frac{c(n,s,a)}{1-s}
	 = \left(\int_\Sf da\right)^{-1}
	 \] 
	 we obtain
	\[
	\lim_{s\to 1} c(n,s,a) \mathcal I_1 =0 .
	\]
On the other hand,
		integrating by parts, we get that
		\bgs{ 
			\int_1^{1+\eps} d\rho  \frac{u(x-r\rho\omega)}{(\rho^2-1)^s \rho}
				= \frac{\eps^{1-s}u\left(x-r(1+\eps)\omega\right)}{(1-s) (\eps+2)^s (1+\eps)}
					-\frac{1}{1-s} \int_1^{1+\eps} d\rho (\rho-1)^{1-s} \frac{d}{d\rho} \frac{u(x-r\rho\omega)}{(\rho+1)^s\rho }
					 .}
		We have that
		\[
			\left| \frac{d}{d\rho} \frac{u(x-r\rho \omega)}{(\rho+1)^s \rho}    \right| \leq c(r) \|u\|_{C^1{(B_{2r}(x))}}, 
		\]
		hence
		\bgs{ 
		\left|\int_1^{1+\eps} \, d\rho  \frac{u(x-r\rho\omega )}{(\rho^2-1)^s\rho} 
				- \frac{\eps^{1-s} u\left(x-r(1+\eps)\omega\right)}{(1-s)(\eps+2)^s (1+\eps)}\right|
				\leq 				\frac{\eps^{2-s} }{1-s} c(r) \|u\|_{C^1{(B_{2r}(x))}}.
	}
	In the same way, we get that
	\bgs{ 
		\left|\int_1^{1+\eps} \, d\rho  \frac{u(x+r\rho\omega )}{(\rho^2-1)^s\rho} 
				- \frac{\eps^{1-s} u\left(x+r(1+\eps)\omega\right)}{(1-s)(\eps+2)^s (1+\eps)}\right|
				\leq 				\frac{\eps^{2-s} }{1-s} c(r) \|u\|_{C^1{(B_{2r}(x))}}.
	}
	It follows that  
		\bgs{
		 \bigg| \mathcal I_2 -&\,   
		 \frac{\eps^{1-s}}{(1-s)(\eps+2)^s (1+\eps)}
	  \int_\Sf da(\omega) \, \big( u\left(x-r(1+\eps)\omega\right) +u\left(x+r(1+\eps)\omega\right)\big)  \bigg|
	\\
	\leq &\, \frac{\eps^{2-s}}{1-s} c(r) \Lambda \|u\|_{C^1(B_{2r}(x))} .
	}
	Multiplying by $c(n,s,a)$ and sending $s\nearrow1$ we get that
	\bgs{
		\lim_{s\nearrow 1} c(n,s,a) \mathcal I_2 = \frac{\left(\int_\Sf da\right)^{-1}  }{ (\eps +2)(\eps+1) } \int_\Sf  da(\omega)\, \big( u(x-r(1+\eps)\omega)+u(x+r(1+\eps)\omega) \big) \, + \mathcal O(\eps).
	}
	For $\eps \to 0$ we get that
	\bgs{
		\lim_{s\nearrow 1} c(n,s,a) \mathcal I_2 = \frac12\left(\int_\Sf da\right)^{-1}   \int_\Sf da(\omega) \, \big(u(x-r\omega) +u(x+r\omega)\big).
		}
			So putting together the limits involving $\mathcal I_1, \mathcal I_2$ into \eqref{opium} we obtain
			 the conclusion.
\end{proof}

\vskip3pt
\noindent 	
We use now the norms introduced at the beginning of Section \ref{weak}. We have the next inequality.
\begin{proposition}\label{ineqq}
Let $u\in H^1_a(\Rn)$. Then there exists $C>0$ independent of $s\in (1/2,1)$ with
\[ 
(1-s)[u]^2_{H_a^s(\Rn)} \leq C  \|u\|^2_{H^1_a(\Rn)}.
\]
\end{proposition}
\begin{proof}
We have that
\bgs{ 
	\int_\Rn dx \int_\Sf da(\omega) \left(u(x)-u(x+\rho\omega) \right)^2 \leq \rho^2 [u]^2_{H^1_a(\Rn)}.
	}
	Indeed, for all $\rho\in\R$, we have
	\bgs{ 
		\int_\Rn dx \int_\Sf da(\omega)\left(u(x)-u(x+\rho\omega) \right)^2 \leq&\;\rho^2 \int_\Rn dx \int_\Sf da(\omega) \left(\int_0^1 \nabla u(x+t\rho\omega)\cdot \omega\,  dt\right)^2
		\\
		\leq&\; \rho^2 \int_\Rn dx \int_\Sf da(\omega) \int_0^1 dt\left( \nabla u(x+t\rho\omega)\cdot \omega\right)^2
		\\
		\leq &\;\rho^2 \int_0^1 dt \int_\Rn dx \int_\Sf da(\omega) \left( \nabla u(x+t\rho\omega)\cdot \omega\right)^2
		\\ = &\; \rho^2  [u]^2_{H^1_a(\Rn)}.
		} 
		Therefore
		\bgs{
		(1-s)\|u\|_{H^s_a(\Rn)}^2\leq&\;2(1-s)[u]^2_{H_a^1(\Rn)} \int_0^1 \rho^{1-2s} \, d\rho 
		\\
		&\;+ (1-s)\int_1^\infty d\rho \rho^{-1-2s} \int_\Rn dx \int_\Sf da(\omega) (u(x)-u(x+\rho\omega))^2 
		\\
		&\;+ (1-s)\int_1^\infty d\rho \rho^{-1-2s} \int_\Rn dx \int_\Sf da(\omega) (u(x)-u(x-\rho\omega))^2  
		\\
		\leq &\; C( [u]^2_{H_a^1(\Rn)} +  \|u\|^2_{L^2(\Rn)}),
		}
	for some positive constant $C$.
\end{proof}
In what follows, we prove a Bourgain-Brezis-Mironescu type property \cite{BBM} for anisotropic norms.
A different type of anisotropicity in the formula was recently investigated in \cite{Anis} and in \cite{Ludwig}.

\begin{proposition}[BBM type formula]
	\label{bbm-th}
Let $u\in H_{a,0}^1(\Rn)$. Then, we have the formula
\eqlab{ \label{bbm}
	\lim_{s\to 1} (1-s)[u]_{H^s_a(\Rn)}^2
	=
	 [u]^2_{H^1_a(\Rn)}.
	}
\end{proposition}
\begin{proof} 
We prove \eqref{bbm} first for any $u\in C^1_c(\Rn)$.
We write
	\eqlab{\label{jjj}
		&\int_{\R^n}dx\int_\R d\rho\int_\Sf da(\omega) \frac{\big(u(x)-u(x+\rho\omega)\big)^2}{\rho^{1+2s}} 
		\\
		=&\,\int_{\R^n}dx\int_0^\infty d\rho\int_\Sf da(\omega) \frac{\big(u(x)-u(x+\rho\omega)\big)^2}{\rho^{1+2s}} 
		\\
		&\;+\int_{\R^n}dx\int_0^\infty d\rho\int_\Sf da(\omega) \frac{\big(u(x)-u(x-\rho\omega)\big)^2}{\rho^{1+2s}} .
	}
Since $u\in C^1$ by the mean value theorem, there exist $\underline h:=\underline h(\rho,\omega), \overline h:=\overline h(\rho,\omega)\in[0,\rho]$ such that
\bgs{
	u(x+\rho\omega)-u(x)=& \;\rho\omega\cdot\nabla u(x+\overline h\omega) \quad \mbox{ and }
	\\
	u(x-\rho\omega)-u(x)=& \;-\rho\omega\cdot\nabla u(x+\underline h\omega).
}
Furthermore, for any $\eps>0$ there exists $r:=r(\eps)>0$ such that 	
	\eqlab{\label{hhhh}
		\left|  \nabla u(x+h\omega)-\nabla u(x) \right| <\eps 
		\qquad \mbox{ whenever } |h|=|h\omega|<\rho<r.
	}
We fix $\eps>0$ (to be taken arbitrarily small in the sequel) and consider the correspondent $r:=r(\eps)$. We then write
\bgs{
\int_{\R^n}dx&\int_0^\infty d\rho\int_\Sf da(\omega) \frac{\big(u(x)-u(x+\rho\omega)\big)^2}{\rho^{1+2s}} 
		\\
		=&\;
				\int_{\R^n}dx\int_0^r d\rho\int_\Sf da(\omega) \rho^{1-2s} \big(\nabla u(x+\overline h\omega) \cdot \omega\big)^2
	\\
	&\;+
		\int_{\R^n}dx\int_r^\infty d\rho\int_\Sf da(\omega)  \frac{\big(u(x)-u(x+\rho\omega)\big)^2}{\rho^{1+2s}} 
		\\
		=:&\; I_{r,s}^1+I_{r,s}^2.
		}	
		Notice that
		\bgs{
			I_{r,s}^1 = &\;	\int_{\R^n} dx\int_0^r d\rho \rho^{1-2s} \int_\Sf da(\omega)    \big(\nabla u(x+\overline h\omega) \cdot \omega\big)^2 -  \big(\nabla u(x) \cdot \omega\big)^2  
			\\
	+&\;
	 	\int_{\R^n}dx\int_0^r d\rho \rho^{1-2s} \int_\Sf da(\omega)  \big(\nabla u(x) \cdot \omega\big)^2 
	 	\\
	 	=&\; J_{r,s}^1 +J_{r,s}^2.
	}
	From \eqref{hhhh}, we have that 
	\bgs{
	 \Big |&\big(\nabla u(x+\overline h\omega) \cdot \omega\big)^2 -  \big(\nabla u(x) \cdot \omega\big)^2   \Big| 
	 \\
	 \leq &\; \Big| \left(\nabla u(x+\overline h \omega) -\nabla u(x) \right)\cdot \omega \Big| \, \Big|\left(\nabla u(x+\overline h \omega) -\nabla u(x) \right) \cdot \omega \Big|
	\leq 2\eps \|u\|_{C^1(\R^n)}.
	}
	Therefore, for some compact set $K\subset\R^n$ independent of $\eps,$ there holds
	\bgs{ \left|J_{r,s}^1 \right| \leq \eps \Lambda \|u\|_{C^1(\R^n)} \frac{r^{2-2s}}{2(1-s)}|K|.
	}
Also	
	\bgs{
		J_{r,s}^2 = \frac{r^{2-2s}}{2(1-s)} \int_{\R^n}dx\int_\Sf da(\omega) (\nabla u(x)\cdot \omega)^2.
	}
	It follows that
	\[ \lim_{s\to 1} (1-s) I_{r,s}^1 =\frac12\int_{\R^n}dx\int_\Sf  da(\omega)(\nabla u(x)\cdot \omega)^2 +\mathcal O(\eps).\]
			Furthermore we get that
		\bgs{\left| I_{r,s}^2\right|
		\leq
		&\; 2\|u\|_{L^2(\Rn)}^2\Lambda \int_r^\infty \rho^{-1-2s}\, d\rho			
		=			\frac{ r^{2s}\|u\|_{L^2(\Rn)}^2 \Lambda }s,
		}
		hence
		\[ \lim_{s\to 1} (1-s) I_{r,s}^2= 0.\]
		We finally get that
		\bgs{
		 \lim_{s\to 1} (1-s) \int_{\R^n}dx\int_0^\infty d\rho\int_\Sf da(\omega) \frac{\big(u(x)-u(x+\rho\omega)\big)^2}{\rho^{1+2s}}  = \frac12\int_\Sf  da(\omega)(\nabla u(x)\cdot \omega)^2 +\mathcal O(\eps).
		}
		We obtain the same limit for the second term in \eqref{jjj} and get \eqref{bbm} for $u\in C_c^1(\Rn)$ by sending $ \eps\to 0$.
	Let now $u\in H_{a,0}^1(\Rn)$. There exists $\{u_j\}_j \in C_c^\infty(\Rn)$ such that
	\[ \|u-u_j\|_{H^1_a(\Rn)} \to 0 \quad \mbox{ as } j\to \infty.\]
	Then, according to Proposition \ref{ineqq} we have that
	\[ (1-s)\left( [u]_{H^s_a(\Rn)}-[u_j]_{H^s_a(\Rn)}  \right)^2 \leq (1-s)[u-u_j]^2_{H^s_a(\Rn)} \leq C \|u-u_j\|^2_{H^1_a(\Rn) } \to 0 \quad \mbox { as } j\to 0.  \]
	The conclusion \eqref{bbm} for $u\in H^1_{a,0}(\Rn)$ immediately follows.
\end{proof}

\begin{remark}
	If $da=d{\mathscr H}^{n-1}$, the left-hand side of the formula in 
	Proposition \ref{bbm} boils down to
	$$
	\lim_{s\to 1} (1-s)\int_{\R^n} dx\int_0^\infty d\rho\int_\Sf d{\mathscr H}^{n-1}(\omega) \frac{\big(u(x)-u(x+\rho\omega)\big)^2}{|\rho|^{1+2s}}, 
	$$
	while the right-hand side to
	$$
		\frac{1}{2} \int_{\R^n}dx \int_\Sf d{\mathscr H}^{n-1}(\omega)(\omega\cdot\nabla u(x))^2=\frac{Q_{n,2}}{2}\int_\Omega |\nabla u(x)|^2 dx
	$$
	where
	$$
	Q_{n,2}=\int_\Sf |\sigma\cdot \omega|^2d{\mathscr H}^{n-1}(\omega)
	$$
	for some $\sigma\in \Sf$. This is consistent with the usual Brezis-Bourgain-Mironescu 
	formula (see \cite{BBM}).
	\end{remark}

\bigskip

\bibliography{biblio}

\begin{thebibliography}{10}

\bibitem{AbatLarge}
Nicola Abatangelo.
\newblock Large {$S$}-harmonic functions and boundary blow-up solutions for the
  fractional {L}aplacian.
\newblock {\em Discrete Contin. Dyn. Syst.}, 35(12):5555--5607, 2015.

\bibitem{gettinacq}
Nicola Abatangelo and Enrico Valdinoci.
\newblock Getting acquainted with the fractional laplacian.
\newblock {\em arXiv preprint arXiv:1710.11567}, 2017.

\bibitem{abramowitz}
Milton Abramowitz and Irene~A. Stegun, editors.
\newblock {\em Handbook of mathematical functions with formulas, graphs, and
  mathematical tables}.
\newblock A Wiley-Interscience Publication. John Wiley \& Sons, Inc., New York;
  National Bureau of Standards, Washington, DC, 1984.
\newblock Reprint of the 1972 edition, Selected Government Publications.

\bibitem{AL}
\'Angel Arroyo and Jos\'e~G. Llorente.
\newblock On the asymptotic mean value property for planar {$p$}-harmonic
  functions.
\newblock {\em Proc. Amer. Math. Soc.}, 144(9):3859--3868, 2016.

\bibitem{BassChen}
Richard~F. Bass and Zhen-Qing Chen.
\newblock Regularity of harmonic functions for a class of singular stable-like
  processes.
\newblock {\em Math. Z.}, 266(3):489--503, 2010.

\bibitem{basskass}
Richard~F. Bass and Moritz Kassmann.
\newblock H\"older continuity of harmonic functions with respect to operators
  of variable order.
\newblock {\em Comm. Partial Differential Equations}, 30(7-9):1249--1259, 2005.

\bibitem{bogdan1}
Krzysztof Bogdan and {Pawe\l} Sztonyk.
\newblock Estimates of the potential kernel and {H}arnack's inequality for the
  anisotropic fractional {L}aplacian.
\newblock {\em Studia Math.}, 181(2):101--123, 2007.

\bibitem{BBM}
Jean Bourgain, Haim Brezis, and Petru Mironescu.
\newblock Another look at {S}obolev spaces.
\newblock In {\em Optimal control and partial differential equations}, pages
  439--455. IOS, Amsterdam, 2001.

\bibitem{bucur}
Claudia Bucur.
\newblock Some observations on the {G}reen function for the ball in the
  fractional {L}aplace framework.
\newblock {\em Communications on Pure and Applied Analysis}, 15(2):657--699,
  2016.

\bibitem{nonlocal}
Claudia Bucur and Enrico Valdinoci.
\newblock Nonlocal diffusion and applications.
\newblock {\em arXiv preprint arXiv:1504.08292}, 2015.
\newblock Accepted for Publication for the Springer Series ``Lecture Notes of
  the Unione Matematica Italiana''.

\bibitem{caffy}
Luis Caffarelli and Luis Silvestre.
\newblock Regularity theory for fully nonlinear integro-differential equations.
\newblock {\em Communications on Pure and Applied Mathematics}, 62(5):597--638,
  2009.

\bibitem{hitch}
Eleonora Di~Nezza, Giampiero Palatucci, and Enrico Valdinoci.
\newblock Hitchhiker's guide to the fractional {S}obolev spaces.
\newblock {\em Bull. Sci. Math.}, 136(5):521--573, 2012.

\bibitem{FerrariMean}
Fausto Ferrari.
\newblock Mean value properties of fractional second order operators.
\newblock {\em Commun. Pure Appl. Anal.}, 14(1):83--106, 2015.

\bibitem{garofalo}
Nicola Garofalo.
\newblock Fractional thoughts.
\newblock {\em arXiv preprint arXiv:1712.03347}, 2017.

\bibitem{kassmann}
Moritz Kassmann, Marcus Rang, and Russell~W Schwab.
\newblock Integro-differential equations with nonlinear directional dependence.
\newblock {\em Indiana Univ. Math. J}, 63(5):1467--1498, 2014.

\bibitem{MPR2}
Bernd Kawohl, Juan Manfredi, and Mikko Parviainen.
\newblock Solutions of nonlinear {PDE}s in the sense of averages.
\newblock {\em J. Math. Pures Appl. (9)}, 97(2):173--188, 2012.

\bibitem{Landkof}
Naum~S. Landkof.
\newblock {\em Foundations of modern potential theory}.
\newblock Springer-Verlag, New York-Heidelberg, 1972.
\newblock Translated from the Russian by A. P. Doohovskoy, Die Grundlehren der
  mathematischen Wissenschaften, Band 180.

\bibitem{LM}
Peter Lindqvist and Juan Manfredi.
\newblock On the mean value property for the {$p$}-{L}aplace equation in the
  plane.
\newblock {\em Proc. Amer. Math. Soc.}, 144(1):143--149, 2016.

\bibitem{Ludwig}
Monica Ludwig.
\newblock Anisotropic fractional sobolev norms.
\newblock {\em Adv. Math.}, 252:150--157, 2014.

\bibitem{MPR}
Juan~J. Manfredi, Mikko Parviainen, and Julio~D. Rossi.
\newblock An asymptotic mean value characterization for {$p$}-harmonic
  functions.
\newblock {\em Proc. Amer. Math. Soc.}, 138(3):881--889, 2010.

\bibitem{Anis}
Hoai-Minh Nguyen and Marco Squassina.
\newblock On anisotropic sobolev spaces.
\newblock {\em Preprint}, 2017.

\bibitem{rosserr}
Xavier Ros-Oton and Joaquim Serra.
\newblock Regularity theory for general stable operators.
\newblock {\em J. Differential Equations}, 260(12):8675--8715, 2016.

\bibitem{rosval}
Xavier Ros-Oton and Enrico Valdinoci.
\newblock The {D}irichlet problem for nonlocal operators with singular kernels:
  convex and nonconvex domains.
\newblock {\em Advances in Mathematics}, 288:732--790, 2016.

\bibitem{Valdy}
Raffaella Servadei and Enrico Valdinoci.
\newblock Weak and viscosity solutions of the fractional {L}aplace equation.
\newblock {\em Publ. Mat.}, 58(1):133--154, 2014.

\bibitem{Silvphd}
Luis Silvestre.
\newblock Regularity of the obstacle problem for a fractional power of the
  {L}aplace operator.
\newblock {\em Comm. Pure Appl. Math.}, 60(1):67--112, 2007.

\bibitem{sztonyk}
{Pawe\l} Sztonyk.
\newblock Regularity of harmonic functions for anisotropic fractional
  {L}aplacians.
\newblock {\em Math. Nachr.}, 283(2):289--311, 2010.

\end{thebibliography}
\bibliographystyle{plain}

\end{document}